\documentclass[12pt,oneside,reqno]{amsart}

  \usepackage{a4wide}
  \usepackage{verbatim}
  \usepackage{amssymb}
  \usepackage{amsmath}
  \usepackage{amsfonts}
  \usepackage{amscd}
  \usepackage{xspace}
  \usepackage{boxedminipage}
  \usepackage{nomencl}

  \usepackage{booktabs}
  \usepackage{multirow}
  \usepackage[figuresright]{rotating}
  \usepackage[format=hang,font=small]{caption} 
  \usepackage[curve,graph]{xypic}
  \usepackage{lscape}
  \usepackage{subfigure}
\usepackage[T1]{fontenc}
\usepackage[utf8]{inputenc}

  \usepackage{stmaryrd}
\def\NN{{\mathbb{N}}}
\def\ZZ{{\mathbb{Z}}}

\def\RR{{\mathbb{R}}}

\def\min{{\mathrm{min}}}

\def\centre{{\mathrm{centre}}}
\def\arm{{\mathrm{arm}}}
\def\shape{{\mathrm{shape}}}
\def\mult{{\mathrm{mult}}}

  \usepackage{tikz}
  \usetikzlibrary{calc}

  \makeatletter
  \CheckCommand*\refstepcounter[1]{\stepcounter{#1}%
      \protected@edef\@currentlabel
       {\csname p@#1\endcsname\csname the#1\endcsname}%
  }
  \renewcommand*\refstepcounter[1]{\stepcounter{#1}%
    \protected@edef\@currentlabel
      {\csname p@#1\expandafter\endcsname\csname the#1\endcsname}%
  }
  \def\labelformat#1{\expandafter\def\csname p@#1\endcsname##1}
  \DeclareRobustCommand\Ref[1]{\protected@edef\@tempa{\ref{#1}}%
     \expandafter\MakeUppercase\@tempa
  }
  \makeatother

  \makeatletter
  \newcommand{\numberlike}[2]{%
     \expandafter\def\csname c@#1\endcsname{%
         \expandafter\csname c@#2\endcsname}%
  }
  \makeatother

  \def\DefaultNumberTheoremWithin{section}

  \theoremstyle{plain}
  \newtheorem{Lemma}{Lemma}
     \numberwithin{Lemma}{\DefaultNumberTheoremWithin}
     \labelformat{Lemma}{Lemma~#1}
  
     \numberwithin{Claim}{\DefaultNumberTheoremWithin}
     \numberlike{Claim}{Lemma}
     \labelformat{Claim}{Claim~#1}
  \newtheorem{Theorem}{Theorem}
     \numberwithin{Theorem}{\DefaultNumberTheoremWithin}
     \numberlike{Theorem}{Lemma}
     \labelformat{Theorem}{Theorem~#1}
  \newtheorem{Corollary}{Corollary}
     \numberwithin{Corollary}{\DefaultNumberTheoremWithin}
     \numberlike{Corollary}{Lemma}
     \labelformat{Corollary}{Corollary~#1}
  \newtheorem{Proposition}{Proposition}
     \numberwithin{Proposition}{\DefaultNumberTheoremWithin}
     \numberlike{Proposition}{Lemma}
     \labelformat{Proposition}{Proposition~#1}
  \newtheorem{Conjecture}{Conjecture}
     \numberwithin{Conjecture}{\DefaultNumberTheoremWithin}
     \numberlike{Conjecture}{Lemma}
     \labelformat{Conjecture}{Conjecture~#1}

  \theoremstyle{definition}
  
     \numberwithin{Definition}{\DefaultNumberTheoremWithin}
     \numberlike{Definition}{Lemma}
     \labelformat{Definition}{Definition~#1}

  \theoremstyle{definition}
  \newtheorem{Question}{Question}
     \numberwithin{Question}{\DefaultNumberTheoremWithin}
     \numberlike{Question}{Lemma}
     \labelformat{Question}{Question~#1}

  \theoremstyle{definition}
  
     \numberwithin{Problem}{\DefaultNumberTheoremWithin}
     \numberlike{Problem}{Lemma}
     \labelformat{Problem}{Problem~#1}

  \theoremstyle{remark}
  
     \numberwithin{Remark}{\DefaultNumberTheoremWithin}
     \numberlike{Remark}{Lemma}
     \labelformat{Remark}{Remark~#1}
  \theoremstyle{remark}

     \numberwithin{Example}{\DefaultNumberTheoremWithin}
     \numberlike{Example}{Lemma}
     \labelformat{Example}{Example~#1}
  
     \labelformat{Case}{Case~#1}
     \numberwithin{Case}{Lemma}
  
     \labelformat{Step}{Step~#1}
     \numberwithin{Step}{Lemma}
     
     \theoremstyle{table}
  
     \numberwithin{Table}{\DefaultNumberTheoremWithin}
     \numberlike{Table}{Lemma}
     \labelformat{Table}{Table~#1}
  \theoremstyle{table}

  \labelformat{equation}{(#1)}
  \labelformat{figure}{Figure~#1}
  \labelformat{section}{\S#1}
  \labelformat{subsection}{\S#1}
  \labelformat{subsubsection}{\S#1}

  \def\eqref{\ref}

  \allowdisplaybreaks

\begin{document}


 \title[Products of stabilizing representations]{Products of stabilizing representations}


  \author{Artur Rapp}
     \address{Fachbereich Mathematik und Informatik\\
              Philipps-Universit\"at Marburg\\
              35032 Marburg\\
              Germany}
     \email{rapp202@mathematik.uni-marburg.de}

  \thanks{}

  \begin{abstract}
We study representation stability in the sense of Church and Farb. We show that products of stabilizing $S_n$-representations fulfill certain recursive relations which can be described by a new class of difference operators.
  \end{abstract}

  \maketitle

\section{Introduction}
We study representations of the symmetric group $S_n$ and formulate our statements about $S_n$-representations in the world of symmetric functions. We refer to Macdonald \cite{M} for background on $S_n$-representations and symmetric functions. The $\ZZ$-module of symmetric functions corresponding to (virtual) $S_n$-representations is denoted by $\Lambda_n$. A number partition $\lambda$ of $n\in \NN$ is a finite weakly decreasing sequence $\lambda=(\lambda_1,...,\lambda_l)$ of positive integers such that $\lambda_1+...+\lambda_l=n$. We also write $\lambda\vdash n$ or $|\lambda|=n$ in that case. The number $l=l(\lambda)$ is the length of $\lambda$. The set of Schur functions $\{s_{\lambda}~|~\lambda\vdash n\}$ is a basis of $\Lambda_n$. We define the componentwise sum of two  partitions $\lambda=(\lambda_1,...,\lambda_l)$ and $\mu=(\mu_1,...,\mu_k)$ with $l\le k$ by
 $\lambda+\mu=(\lambda_1+\mu_1,...,\lambda_l+\mu_l,\mu_{l+1},...,\mu_{k})$. 
 For a fixed partition $\lambda$ we denote by $s_{\mu}+\lambda$ the function $s_{\mu+\lambda}$ and extend this definition from the basis of Schur functions linearly to all symmetric functions. By $\Lambda_{\NN}^{n_0,k}$ we denote the $\ZZ$-module of sequences $\{f_n\}_{n\in\NN}$ with $f_n\in\Lambda_{nk+n_0}$ for all $n\in \NN$. For every $n_0\in\NN$, every partition $\lambda$ and every divisor $m$ of $|\lambda|$ we define
 $$ \Delta_m^{\lambda}:\Lambda_{\NN}^{n_0,|\lambda|/m}\rightarrow \Lambda_{\NN}^{n_0,|\lambda|/m},~
  \{f_n\}_{n\ge 0}\mapsto \{\Delta_m^{\lambda}f_{n+m}\}_{n\ge 0} $$
 $$ \text{where $\Delta_m^{\lambda}f_{n}=f_n-(f_{n-m}+\lambda)$ for all $n\ge m$}. $$
 We write $\Delta^{\lambda}$ for $\Delta^{\lambda}_1$. For a partition $\lambda=(\lambda_1,...,\lambda_l)$ and $n\ge \lambda_1$ we write$(n,\lambda)$ for $(n,\lambda_1,\lambda_2,...,\lambda_l)$. We consider sequences of Schur functions of the form $\{s_{(n,\lambda)}\}_{n\ge\lambda_1}\in \Lambda_{\NN}^{|\lambda|,1}$. Let $\alpha_1,...,\alpha_k\in\NN$ and $\lambda_1,...,\lambda_k$ be number partitions. Then the sequence of products $\{s_{(n+\alpha_1,\lambda_1)}\cdots s_{(n+\alpha_k,\lambda_k)}\}_n$ is an element of $\Lambda_{\NN}^{\alpha_1+|\lambda_1|+\cdots \alpha_k+|\lambda_k|,k}$. We can apply difference operators $\Delta_{|\mu|/k}^{\mu}$ on it where $|\mu|$ is a multiple of $k$. Now, we are in position to formulate our main theorem.
\begin{Theorem}\label{main2}
Let $\alpha_1\ge\alpha_2\ge \alpha_3\ge \ge 0$ and $\lambda_1,\lambda_2,\lambda_3$ be number partitions. Then
\begin{itemize}
\item[(a)]
$$ \Delta^{(2)}\Delta^{(1,1)}(s_{(n+\alpha_1,\lambda_1)}s_{(n+\alpha_2,\lambda_2)})=0~\text{for all $n>|\lambda_1|+|\lambda_2|+1-\alpha_2$}.  $$
The set $\{\Delta^{(2)},\Delta^{(1,1)}\}$ is minimal in the sense that the above sequence is not eventually zero if we remove one of the difference operators.
\item[(b)]
$$ \Delta_2^{(3,3)}\Delta^{(3)}\Delta^{(2,1)}\Delta^{(1,1,1)}(s_{(n+\alpha_1,\lambda_1)}s_{(n+\alpha_2,\lambda_2)}s_{(n+\alpha_3,\lambda_3)})=0$$  $$\text{for all $n>\max\{4,\alpha_1-\alpha_2+l(\lambda_1)\}+2(|\lambda_1|+|\lambda_2|+|\lambda_3|+1)-\alpha_3$}.  $$
The set $\{\Delta_2^{(3,3)},\Delta^{(3)},\Delta^{(2,1)},\Delta^{(1,1,1)}\}$ is minimal in the sense that the above sequence is not eventually zero if we remove one of the difference operators.
\end{itemize}
\end{Theorem}
There is experimental evidence that an analogous statement about fourfold products holds. We formulate this in the following conjecture. We do not provide a complete proof of this statement but show how our methods indicate its validity until reaching a point where the number of cases to discuss is massive.   
\begin{Conjecture}\label{vierfach}
Let $\alpha_1\ge\alpha_2\ge \alpha_3\ge \alpha_4\ge 0$ and $\lambda_1,\lambda_2,\lambda_3,\lambda_4$ be number partitions. Then
$$ \Delta_3^{(4,4,4)}\Delta_2^{(3,3,2)}\Delta^{(4)}\Delta^{(3,1)}(\Delta^{(2,2)})^2(\Delta^{(2,1,1)})^2\Delta^{(1,1,1,1)}
 (s_{(n+\alpha_1,\lambda_1)}s_{(n+\alpha_2,\lambda_2)}s_{(n+\alpha_3,\lambda_3)}s_{(n+\alpha_4,\lambda_4)})=0  $$
 $$ \text{for sufficiently large $n$}. $$
 The multiset $\{\Delta_3^{(4,4,4)},\Delta_2^{(3,3,2)},\Delta^{(4)},\Delta^{(3,1)},\Delta^{(2,2)},\Delta^{(2,2)},\Delta^{(2,1,1)},\Delta^{(2,1,1)},\Delta^{(1,1,1,1)}\}$ is minimal in the sense that the above sequence is not eventually zero if we remove one of the difference operators.
\end{Conjecture}
We show in \ref{commute} that every difference operator $\Delta_m^{\lambda}$ is linear so that we can expand the theorem to a wider set of symmetric function sequences.
A sequence $\{f_n\}_n\in\Lambda_{\NN}^{n_0,1}$ stabilizes at $N\in \NN$ in the sense of Church and Farb \cite{CF} if 
$$ \Delta^{(1)}(f_n)=0~\text{for all $n>N$}. $$
We get the following corollary.
\begin{Corollary}\label{main}
Let $m_1,m_2,m_3,N_1,N_2,N_3\in \NN$ and $\{f^{(1)}_n\}_n,\{f^{(2)}_n\}_n,\{f^{(3)}_n\}_n$ be sequences such that $\{f^{(i)}_n\}_n\in \Lambda_{\NN}^{m_i,1}$ stabilizes at $N_i$ for all $i\in\{1,2,3\}$. Then
\begin{itemize}
\item[(a)]
$$ \Delta^{(2)}\Delta^{(1,1)}(f^{(1)}_nf^{(2)}_n)=0 $$
$$\text{for all $n>\max\{N_1,N_2,m_1+m_2+N_1+N_2+\max\{N_1,N_2\}-2\}$}.  $$
The set $\{\Delta^{(2)},\Delta^{(1,1)}\}$ is minimal in the sense that the above sequence is not eventually zero if we remove one of the difference operators.
\item[(b)]
$$ \Delta_2^{(3,3)}\Delta^{(3)}\Delta^{(2,1)}\Delta^{(1,1,1)}(f^{(1)}_nf^{(2)}_nf^{(3)}_n)=0  $$
$$ \text{for all $n>\max\{N_1,N_2,2(m_1+m_2+m_3+\max\{m_1,m_2,m_3\})+3(N_1+N_2+N_3)-7\}$}. $$
The set $\{\Delta_2^{(3,3)},\Delta^{(3)},\Delta^{(2,1)},\Delta^{(1,1,1)}\}$ is minimal in the sense that the above sequence is not eventually zero if we remove one of the difference operators.
\end{itemize}
\end{Corollary} 
In the same way, \ref{vierfach} is equivalent to the following statement.
\begin{Conjecture}
Let $m_1,...,m_4,N_1,...,N_4\in\NN$ and $\{f^{(1)}_n\}_n,\{f^{(2)}_n\}_n,\{f^{(3)}_n\}_n,\{f^{(4)}_n\}_n$ be sequences such that $\{f^{(i)}_n\}_n\in \Lambda_{\NN}^{m_i,1}$ stabilizes at $N_i$ for all $i\in\{1,2,3,4\}$. Then
$$ \Delta_3^{(4,4,4)}\Delta_2^{(3,3,2)}\Delta^{(4)}\Delta^{(3,1)}(\Delta^{(2,2)})^2(\Delta^{(2,1,1)})^2\Delta^{(1,1,1,1)}\big(f^{(1)}_nf^{(2)}_nf^{(3)}_nf^{(4)}_n\big)=0  $$
 $$ \text{for sufficiently large $n$}. $$
 The multiset $\{\Delta_3^{(4,4,4)},\Delta_2^{(3,3,2)},\Delta^{(4)},\Delta^{(3,1)},\Delta^{(2,2)},\Delta^{(2,2)},\Delta^{(2,1,1)},\Delta^{(2,1,1)},\Delta^{(1,1,1,1)}\}$ is minimal in the sense that the above sequence is not eventually zero if we remove one of the difference operators.
\end{Conjecture}
We formulate a statement about products of arbitrariely many stabilizing sequences as a question.
\begin{Question}
Let $k\ge 1$ and  $\{f^{(1)}_n\}_n\in\Lambda_{\NN}^{m_1,1},...,\{f^{(k)}_n\}_n\in\Lambda_{\NN}^{m_k,1}$ be stabilizing sequences. Let $\{\lambda_1,...,\lambda_r\}$ be the set of partitions of the numbers $k,2k,...,(k-1)k$. Is there a set of nonnegative integers $\{q_1,...,q_{k-1}\}$ and a number $N$ such that
$$ (\Delta_{|\lambda_1|/k}^{\lambda_1})^{q_1}\cdots (\Delta_{|\lambda_r|/k}^{\lambda_r})^{q_r} \left(f^{(1)}_n\cdots f^{(k)}_n \right)=0~\text{for all $n>N$} $$
and $q_1+...+q_{k-1}$ is minimal with this property
and how can we compute the numbers $q_1,...,q_r$ and $N$?
\end{Question}
By \ref{commute} (ii) this is equivalent to 
\begin{Question}
Let $k\ge 1$, $\alpha_1\ge \dots \ge \alpha_4\ge 0$ and $\mu_1,\dots ,\mu_k$ be number partitions. We consider the sequence $\{s_{(n+\alpha_1,\mu_1)}\cdots s_{(n+\alpha_k,\mu_k)}\}_n\in\Lambda_{\NN}^{\alpha_1+|\mu_1|+\cdots +\alpha_k+|\mu_k|,k}$. Let $\{\lambda_1,...,\lambda_r\}$ be the set of partitions of the numbers $k,2k,...,(k-1)k$. Is there a set of nonnegative integers $\{q_1,...,q_{k-1}\}$ and a number $N$ such that
$$ (\Delta_{|\lambda_1|/k}^{\lambda_1})^{q_1}\cdot (\Delta_{|\lambda_r|/k}^{\lambda_r})^{q_r} \left(s_{(n+\alpha_1,\mu_1)}\cdots s_{(n+\alpha_k,\mu_k)} \right)=0~\text{for all $n>N$} $$
and $q_1+...+q_{k-1}$ is minimal with this property
and how can we compute the numbers $q_1,...,q_r$ and $N$?
\end{Question}
The rest of the paper is dedicated to the proof of \ref{main2} and \ref{main}.
\section{Reduction to homogeneous symmetric functions}
We show in the following lemma that the difference operators commute such that we are free to choose their order and that they are linear.  
\begin{Lemma}\label{commute}
Let $\lambda$ and $\mu$ be partitions, $m$ a divisor of $|\lambda|$ and $l$ a divisor of $|\mu|$ with $|\mu|/l=|\lambda|/m$. Let $n_0\ge 0$ and $\{f_n\}_n\in\Lambda_{\NN}^{n_0,|\lambda|/m}$. Then 
\begin{itemize}
\item[(i)]
$$ \Delta^{\lambda}_m\Delta^{\mu}_l(f_n)=\Delta^{\mu}_l\Delta^{\lambda}_m(f_n)=~\text{for all $n$}. $$
\item[(ii)] The map $\Delta_m^{\lambda}:\Lambda_{\NN}^{n_0,|\lambda|/m}\rightarrow \Lambda_{\NN}^{n_0,|\lambda|/m}$ is linear.
\end{itemize}
\end{Lemma} 
\begin{proof}
(i) 
We have
$$  \Delta^{\lambda}_m\Delta^{\mu}_lf_n=\Delta^{\lambda}_m(f_n-(f_{n-l}+\mu))=f_n-(f_{n-l}+\mu)-(f_{n-m}+\lambda)+(f_{n-l-m}+\mu+\lambda) $$
$$ =\Delta^{\mu}_l(f_n-(f_{n-m}+\lambda))=\Delta^{\mu}_l\Delta^{\lambda}_mf_n. $$
(ii) $\Delta_m^{\lambda}$ is the difference of the shift operator $\{f_n\}_n\mapsto \{f_{n+m}\}_n$ which is linear and the map $\{f_n\}_n \mapsto \{f_n+\lambda\}_n$ which is a linear extension.
\end{proof}
It follows from \ref{commute} (i) that we can define iterated products $\prod_{\Delta\in D} \Delta$ over sets of difference operators $D=\{\Delta_{m_1}^{\lambda_1},\dots, \Delta_{m_r}^{\lambda_r}\}$.
Now, we prove \ref{main} using \ref{main2} and \ref{commute} (ii).
\begin{proof}
(a) The sequence $\{f_n^{(i)}\}_n\in \Delta_{\NN}^{m_i,1}$ stabilizes at $N_i$ for every $i\in\{1,2\}$. It follows that $f_{n}^{(i)}$ is a linear combination of Schur functions $s_{(n+\alpha_i,\lambda_i)}$ with $\alpha_i\ge 1-N_i$ and $\lambda_i\vdash m_i-\alpha_i$ for all $n\ge N_i$. \ref{main2} yields that
$$ \Delta^{(2)}\Delta^{(1,1)}(s_{(n+\alpha_1,\lambda_1)}s_{(n+\alpha_2,\lambda_2)})=0~\text{for all $n>|\lambda_1|+|\lambda_2|+1-\min\{\alpha_1,\alpha_2\}$}. $$
We have
$$ |\lambda_1|+|\lambda_2|+1-\min\{\alpha_1,\alpha_2\}\le m_1+m_2+N_1+N_2+\max\{N_1,N_2\}-2. $$
The claim follows from \ref{commute} (ii).\\
(b) The sequence $\{f_n^{(i)}\}_n\in \Delta_{\NN}^{m_i,1}$ stabilizes at $N_i$ for every $i\in\{1,2,3\}$. It follows that $f_{n}^{(i)}$ is a linear combination of Schur functions $s_{(n+\alpha_i,\lambda_i)}$ with $m_i\ge\alpha_i\ge 1-N_i$ and $\lambda_i\vdash m_i-\alpha_i$ for all $n\ge N_i$. Suppose $\alpha_1\ge\alpha_2\ge \alpha_3$. \ref{main2} yields that
$$ \Delta_2^{(3,3)}\Delta^{(3)}\Delta^{(2,1)}\Delta^{(1,1,1)}(s_{(n+\alpha_1,\lambda_1)}s_{(n+\alpha_2,\lambda_2)}s_{(n+\alpha_3,\lambda_3)})=0$$  $$\text{for all $n>\max\{4,\alpha_1-\alpha_2+l(\lambda_1)\}+2(|\lambda_1|+|\lambda_2|+|\lambda_3|+1)-\alpha_3$}. $$
We have
$$ \max\{4,\alpha_1-\alpha_2+l(\lambda_1)\}+2(|\lambda_1|+|\lambda_2|+|\lambda_3|+1)-\alpha_3 $$
$$ \le m_1+N_2-1+m_1+N_1-1+2(m_1+m_2+m_3+N_1+N_2+N_3-2)+N_3-1 $$
$$ =4m_1+2m_2+2m_3+3(N_1+N_2+N_3)-7 $$
The claim follows from \ref{commute} (ii).
\end{proof}
We want to show next that we can restrict to products of homogeneous symmetric functions $s_{(n)}$ if we additionally multiply with a constant sequence. 
 \begin{Lemma}\label{fromsntosnlambda}
 Let $k\in \NN$ and $\lambda_1,...,\lambda_m$ be number partitions of multiples of $k$. 
 If for all number partitions $\beta$ and $\alpha_1,...,\alpha_{k-1}\in \NN$ the sequence $\{s_{(n+\alpha_1)}\cdots s_{(n+\alpha_{k-1})}s_{(n)}s_{\beta}\}_n\in \Lambda_{\NN}^{\alpha_1+\cdots +\alpha_{k-1}+|\beta|,k}$ fulfills
 $$ \Delta^{\lambda_1}_{|\lambda_1|/k}...\Delta^{\lambda_m}_{|\lambda_m|/k} (s_{(n+\alpha_1)}\cdots s_{(n+\alpha_{k-1})}s_{(n)}s_{\beta})=0 $$
 for all $n$ greater than some number $n_0(\alpha_1,...,\alpha_{k-1},|\beta|)$ then 
  for all number partitions $\mu_1,...,\mu_k$ and $\alpha_1,...,\alpha_{k-1}\in \NN$ the sequence\\
 $\{s_{(n+\alpha_1,\mu_1)}\cdots s_{(n+\alpha_{k-1},\mu_{k-1})}s_{(n,\mu_k)}\}_n\in \Lambda_{\NN}^{\alpha_1+\cdots +\alpha_{k-1}+|\mu_1|+\cdot +|\mu_k|,k}$ fulfills
 $$ \Delta^{\lambda_1}_{|\lambda_1|/k}...\Delta^{\lambda_m}_{|\lambda_1|/k} (s_{(n+\alpha_1,\mu_1)}\cdots s_{(n+\alpha_{k-1},\mu_{k-1})}s_{(n,\mu_k)})=0 $$
 for all $n>\max\{n_0(\alpha_1+i_1,...,\alpha_k+i_k,|\mu_1|-i_1+...+|\mu_k|-i_k)~|~i_q\in\{0,...,l(\mu_q)\}~\text{for all $q\in \{1,...,k\}$}\}$.
 \end{Lemma}
 \begin{proof}
 Let $\alpha_k=0$.
 The Jacobi-Trudi identity yields
 $$ \prod_{q=1}^k s_{(n+\alpha_q,\mu_q)} $$
 $$ =\prod_{q=1}^k \det\begin{pmatrix}
 s_{(n+\alpha_q)} & s_{(n+\alpha_q+1)} & ... & s_{(n+\alpha_q+l(\mu_q))}\\
 s_{(\mu_{q,1}-1)} & s_{(\mu_{q,1})} & ... & s_{(\mu_{q,1}+l(\mu_q)-1)}\\
 ... & ... & ... & ...\\
 s_{(\mu_{q,l(\mu_q)}-l(\mu_q))} & ... & ... & s_{(\mu_{q,l(\mu_q)})}
 \end{pmatrix} $$
 $$ =\prod_{q=1}^k\left(\sum_{i=0}^{l(\mu_q)}s_{(n+\alpha_q+i)}(-1)^i \det(M_{\mu_q,i})\right)
  $$
 $$ = \sum_{i_1=0}^{l(\mu_1)}...\sum_{i_k=0}^{l(\mu_k)}
 \left(\prod_{q=1}^k s_{(n+\alpha_q+i_q)}\right)
 (-1)^{i_1+...+i_k}\prod_{q=1}^k\det(M_{\mu_q,i_q}) $$
 where for $\gamma\in \{\mu_1,...,\mu_k\}$ the matrix $M_{\gamma,i}$ is the matrix we get by deleting the $i$th column of
 $$\begin{pmatrix}
 s_{(\gamma_1-1)} & s_{(\gamma_1)} & ... & s_{(\gamma_1+l(\gamma)-1)}\\
 ... & ... & ... & ...\\
 s_{(\gamma_{l(\gamma)}-l(\gamma))} & ... & ... & s_{(\gamma_{l(\gamma)})}
 \end{pmatrix}. $$
The degree of $\prod_{q=1}^k\det(M_{\mu_q,i_q})$ is $|\mu_1|-i_1+...+|\mu_k|-i_k$. It follows from the assumption that this sequence
vanishes under $\Delta^{\lambda_1}_{n_1}...\Delta^{\lambda_m}_{n_m}$ for $n>\max\{n_0(\alpha_1+i_1,...,\alpha_k+i_k,|\mu_1|-i_1+...+|\mu_k|-i_k)~|~i_q\in\{0,...,l(\mu_q)\}~\text{for all $q\in \{1,...,k\}$}\}$.
 \end{proof}
 It follows that to prove \ref{main2} it is sufficient to prove the following
 \begin{Proposition}\label{main3}
Let $\alpha_1,\alpha_2,\alpha_3\in \NN$ and $\beta$ be a number partition. 
\begin{itemize}
\item[(a)]
The sequence $\{s_{(n+\alpha_1)}s_{(n)}s_{\beta}\}_n\in\Lambda_{\NN}^{\alpha_1+|\beta|,2}$ fulfills
$$ \Delta^{(2)}\Delta^{(1,1)}(s_{(n+\alpha_1)}s_{(n)}s_{\beta})=0~\text{for all $n>\beta_1+1$}.  $$
\item[(b)]
The sequence $\{s_{(n+\alpha_1)}s_{(n+\alpha_2)}s_{(n)}s_{\beta}\}_n\in\Lambda_{\NN}^{\alpha_1+\alpha_2+|\beta|,3}$ fulfills
$$ \Delta_2^{(3,3)}\Delta^{(3)}\Delta^{(2,1)}\Delta^{(1,1,1)}(s_{(n+\alpha_1)}s_{(n+\alpha_2)}s_{(n)}s_{\beta})=0 $$ 
$$\text{for all $n>\max\{4,\alpha_1-\alpha_2+\beta_1+2\}+\beta_1$}.  $$
\item[(c)]
The sequence $\{s_{(n+\alpha_1)}s_{(n+\alpha_2)}s_{(n+\alpha_3)}s_{(n)}s_{\beta}\}_n\in\Lambda_{\NN}^{\alpha_1+\alpha_2+\alpha_3+|\beta|,4}$ fulfills
$$ \Delta_3^{(4,4,4)}\Delta_2^{(3,3,2)}\Delta^{(4)}\Delta^{(3,1)}(\Delta^{(2,2)})^2(\Delta^{(2,1,1)})^2\Delta^{(1,1,1,1)}(s_{(n+\alpha_1)}
s_{(n+\alpha_2)}s_{(n+\alpha_3)}s_{(n)}s_{\beta})=0  $$
$$ \text{for sufficiently large $n$}. $$
\end{itemize}
\end{Proposition} 
For every symmetric function $f$ and $l\in \NN$ we write $f_{\le l}$ for the part of the Schur function decomposition of $f$ with partitions of length less than or equal to $l$ and $f_{>l}$ for the part of the Schur function decomposition with partitions of length greater than $l$.
\begin{Lemma}
Let $k,n_0,l\in \NN$ and $\{f_n\}_n\in \Lambda_{\NN}^{n_0,k}$. Let $\lambda_1,...,\lambda_m$ be number partitions of multiples of $k$. Suppose
$$ \Delta^{\lambda_1}_{|\lambda_1|/k}...\Delta^{\lambda_m}_{|\lambda_m|/k} (f_n)=0. $$
Then $$ \left(\prod_{i:l(\lambda_i)\le l} \Delta^{\lambda_i}_{|\lambda_i|/k}\right) (f_{n})_{\le l}=0. $$  
\end{Lemma}
\begin{proof}
We have
$$ \Delta^{\lambda_1}_{|\lambda_1|/k}...\Delta^{\lambda_m}_{|\lambda_m|/k} f_n=\Delta^{\lambda_1}_{|\lambda_1|/k}...\Delta^{\lambda_m}_{|\lambda_m|/k}(f_n)_{\le l}+\Delta^{\lambda_1}_{|\lambda_1|/k}...\Delta^{\lambda_m}_{|\lambda_m|/k} (f_n)_{>l} $$
$$ = \left(\prod_{i:l(\lambda_i)\le l} \Delta^{\lambda_i}_{|\lambda_i|/k}\right)\left(\prod_{i:l(\lambda_i)> l} \Delta^{\lambda_i}_{|\lambda_i|/k}\right) (f_{n})_{\le l}+\Delta^{\lambda_1}_{|\lambda_1|/k}...\Delta^{\lambda_m}_{|\lambda_m|/k} (f_n)_{>l}.$$
We can write 
$$ \left(\prod_{i:l(\lambda_i)> l} \Delta^{\lambda_i}_{|\lambda_i|/k}\right) (f_{n})_{\le l}=(f_{n})_{\le l}+g_n. $$
for a function $g_n$ with only partitions of length greater than $l$ in its Schur function decomposition. It follows
$$ 0=\Delta^{\lambda_1}_{|\lambda_1|/k}...\Delta^{\lambda_m}_{|\lambda_m|/k} f_n=\left(\prod_{i:l(\lambda_i)\le l} \Delta^{\lambda_i}_{|\lambda_i|/k}\right)(f_{n})_{\le l}+\left(\prod_{i:l(\lambda_i)\le l} \Delta^{\lambda_i}_{|\lambda_i|/k}\right)g_n+\Delta^{\lambda_1}_{|\lambda_1|/k}...\Delta^{\lambda_m}_{|\lambda_m|/k} (f_n)_{>l}. $$
All partitions with length less than or equal to $l$ appear in $\left(\prod_{i:l(\lambda_i)\le l} \Delta^{\lambda_i}_{|\lambda_i|/k}\right)(f_{n})_{\le l}$ while all partitions with length greater than $l$ appear in $\left(\prod_{i:l(\lambda_i)\le l} \Delta^{\lambda_i}_{|\lambda_i|/k}\right)g_n+\Delta^{\lambda_1}_{|\lambda_1|/k}...\Delta^{\lambda_m}_{|\lambda_m|/k} (f_n)_{>l}$ and it follows that each of these two parts must itself be zero.
\end{proof}
Consider a semistandard skew tableau $T$ of shape $\nu/\beta$ and weight $(n+\alpha_1,n+\alpha_2,...,n+\alpha_k)$. We split $T$ into two parts: The part of the first $\beta_1$ columns which we call the centre of $T$ and denote it $\centre(T)$ and the rest which we call the arm of $T$ and denote it $\arm(T)$. For example, let
$$T=
\begin{tikzpicture}[scale=0.5, line width=1pt]
  \draw (0,0) grid (6,1);
  \draw (0,-1) grid (4,1);
  \draw (0,-2) grid (2,1);
  \draw (0,-3) grid (1,1);
  \node[left] (A) at (3,0.5) {1};
  \node[left] (A) at (4,0.5) {1};
  \node[left] (A) at (5,0.5) {1};
  \node[left] (A) at (6,0.5) {2};
  \node[left] (A) at (2,-0.5) {1};
  \node[left] (A) at (3,-0.5) {2};
  \node[left] (A) at (4,-0.5) {3};
  \node[left] (A) at (1,-1.5) {2};
  \node[left] (A) at (2,-1.5) {2};
  \node[left] (A) at (1,-2.5) {3};
\end{tikzpicture}.
$$ 
Then
$$
\centre(T)=
\begin{tikzpicture}[scale=0.5, line width=1pt]
  \draw (0,0) grid (2,1);
  \draw (0,-1) grid (2,1);
  \draw (0,-2) grid (2,1);
  \draw (0,-3) grid (1,1);
  \node[left] (A) at (2,-0.5) {1};
  \node[left] (A) at (1,-1.5) {2};
  \node[left] (A) at (2,-1.5) {2};
  \node[left] (A) at (1,-2.5) {3};
\end{tikzpicture}~\text{and}~\arm(T)=
\begin{tikzpicture}[scale=0.5, line width=1pt]
  \draw (0,0) grid (4,1);
  \draw (0,-1) grid (2,1);
  \node[left] (A) at (1,0.5) {1};
  \node[left] (A) at (2,0.5) {1};
  \node[left] (A) at (3,0.5) {1};
  \node[left] (A) at (4,0.5) {2};
  \node[left] (A) at (1,-0.5) {2};
  \node[left] (A) at (2,-0.5) {3};
\end{tikzpicture}.
$$
For fixed $\beta$ and $k$, there are only finitely many tableaux appearing as centres of tableaux of shape $\nu/\beta$ for arbitrary $\nu$ and we denote this finite set by $c_k(\beta)$. We denote the set of semistandard Young tableaux of weight $\gamma$ and arbitrary shape by $ST(\gamma)$ and the number of occurences of the Symbol $i$ in the tableau $c\in c_k(\beta)$ by $c(i)$
\begin{Lemma}
Let $\alpha_1,...,\alpha_k\in \NN$ and $\beta$ be a number partition. Then
$$ s_{(n+\alpha_1)}...s_{(n+\alpha_k)}s_{\beta}=\sum_{c\in c_k(\beta)}(s_{(n+\alpha_1-c(1))}...s_{(n+\alpha_k-c(k))})_{\le \max\{i~|~\shape(c)_i=\beta_1\}}+\shape(c) $$
for all $n\in \NN$.
\end{Lemma}
\begin{proof}
We have
$$ s_{(n+\alpha_1)}...s_{(n+\alpha_k)}s_{\beta}=\sum_{T'}s_{\shape(T')} $$
where the sum runs over all semistandard skew tableau $T'$ of shape $\nu/\beta$ for any $\nu$ and weight $(n+\alpha_1,n+\alpha_2,...,n+\alpha_k)$
We can rewrite this sum by splitting every such skew tableau into its centre and arm:
$$\sum_{T'}s_{\shape(T')}=\sum_{c\in c_k(\beta)}\sum_{T}s_{\shape(c)+\shape(T)}=\sum_{c\in c_k(\beta)}\left(\left(\sum_{T}s_{\shape(T)}\right)+\shape(c)\right) $$
where the sum runs over all $T\in ST((n+\alpha_1-c(1),...,n+\alpha_k-c(k)))$ such that $T$ has as most as many rows as $c$ has rows of length $\beta_1$. Note that $\sum_{T}s_{\shape(T)}$ is the part of $s_{(n+\alpha_1-c(1))}...s_{(n+\alpha_k-c(k))}$ with Schur functions with partitions with at most as many rows as $c$ has rows of length $\beta_1$. 
\end{proof}
The following lemma follows from the previous two lemmas.
\begin{Lemma}
Let $k\in\NN$ and $\beta$ be a number partition. Let $\lambda_1,...,\lambda_m$ be number partitions of multiples of $k$.
If for all numbers $\alpha_1,...,\alpha_k\in \NN$ there is a number $N(\alpha_1,...,\alpha_k)$ such that the sequence $\{s_{(n+\alpha_1)}...s_{(n+\alpha_k)}\}\in\Lambda_{\NN}^{\alpha_1+...+\alpha_k,k}$ fulfills
$$ \Delta^{\lambda_1}_{n_1}\cdots\Delta^{\lambda_m}_{n_m}(s_{(n+\alpha_1)}\cdots s_{(n+\alpha_k)})=0~\text{for all $n>N(\alpha_1,...,\alpha_k)$} $$
then the sequence $\{s_{(n+\alpha_1)}...s_{(n+\alpha_k)}s_{\beta}\}\in\Lambda_{\NN}^{\alpha_1+...+\alpha_k+|\beta|,k}$ fulfills
$$ \Delta^{\lambda_1}_{n_1}\cdots\Delta^{\lambda_m}_{n_m}(s_{(n+\alpha_1)}\cdots s_{(n+\alpha_k)}s_{\beta})=0$$ $$
\text{for all $n>\max\{N(\alpha_1-c(1),...,\alpha_k-c(k))~|~c\in c_k(\beta)\}+\beta_1$}. $$
\end{Lemma}
This lemma shows that it is sufficient to prove the following proposition.
\begin{Proposition}\label{main4}
Let $\alpha_1\ge \alpha_2\ge 0$. 
\begin{itemize}
\item[(a)]
The sequence $\{s_{(n+\alpha_1)}s_{(n)}\}_n\in \Lambda_{\NN}^{\alpha_1,2}$ fulfills
$$ \Delta^{(2)}\Delta^{(1,1)}s_{(n+\alpha_1)}s_{(n)}=0~\text{for all $n>1$}.  $$
The set $\{\Delta^{(2)},\Delta^{(1,1)}\}$ is minimal in the sense that the above sequence is not eventually zero if we remove one of the difference operators.
\item[(b)]
The sequence $\{s_{(n+\alpha_1)}s_{(n+\alpha_2)}s_{(n)}\}_n\in \Lambda_{\NN}^{\alpha_1+\alpha_2,3}$ fulfills
$$ \Delta_2^{(3,3)}\Delta^{(3)}\Delta^{(2,1)}\Delta^{(1,1,1)}s_{(n+\alpha_1)}s_{(n+\alpha_2)}s_{(n)}=0 $$  $$\text{for all $n>\max\{4,\alpha_1-\alpha_2+2\}$}.  $$
The set $\{\Delta_2^{(3,3)},\Delta^{(3)},\Delta^{(2,1)},\Delta^{(1,1,1)}\}$ is minimal in the sense that the above sequence is not eventually zero if we remove one of the difference operators.
\end{itemize}
\end{Proposition} 
We do not provide a complete proof of the following statement about fourfold products but show how our methods point to its validity until reaching a point where the number of cases to discuss is massive. 
\begin{Conjecture}\label{vierfachhom}
Let $\alpha_1\ge \alpha_2\ge \alpha_3\ge 0$. 
The sequence $\{s_{(n+\alpha_1)}s_{(n+\alpha_2)}s_{(n+\alpha_3)}s_{(n)}\}_n\in \Lambda_{\NN}^{\alpha_1+\alpha_2+\alpha_3,4}$ fulfills
$$ \Delta_3^{(4,4,4)}\Delta_2^{(3,3,2)}\Delta^{(4)}\Delta^{(3,1)}(\Delta^{(2,2)})^2(\Delta^{(2,1,1)})^2\Delta^{(1,1,1,1)}s_{(n+\alpha_1)}
s_{(n+\alpha_2)}s_{(n+\alpha_3)}s_{(n)}=0  $$
$$ \text{for sufficiently large $n$.} $$
 The multiset $\{\Delta_3^{(4,4,4)},\Delta_2^{(3,3,2)},\Delta^{(4)},\Delta^{(3,1)},\Delta^{(2,2)},\Delta^{(2,2)},\Delta^{(2,1,1)},\Delta^{(2,1,1)},\Delta^{(1,1,1,1)}\}$ is minimal in the sense that the above sequence is not eventually zero if we remove one of the difference operators.
\end{Conjecture}

\section{Proof of \ref{main4}}
In the whole section $n$ and $\alpha_1\ge\alpha_2 \ge \alpha_3$ are natural numbers.
In the next lemmas, we use partial matrices of the form $$\begin{pmatrix}
 a_{11} & a_{12} & a_{13} & \dots & a_{1k}\\
 a_{21} &  a_{22} & \dots & a_{2,(k-1)}\\
 \vdots\\
 a_{(k-1),1} & a_{(k-1),2}\\
 a_{k,1}  
\end{pmatrix}.$$
For every $i\in \{1,...,k-1\}$, we write $a_i$ for the $i$th row of $a$ and $|a_i|$ for the sum of the entries of the $i$th row. Let $k,n\in\NN$ and $\alpha=(\alpha_1,...,\alpha_k)\in \NN_0^k$ with $\alpha_1\ge...\ge \alpha_k$. Let $P_{k,n,\alpha}$ be the set of all partial matrices
$  \begin{pmatrix}
 a_{11} & a_{12} & a_{13} & \dots & a_{1k}\\
 a_{21} &  a_{22} & \dots & a_{2,(k-1)}\\
 \vdots\\
 a_{(k-1),1} & a_{(k-1),2}\\
 a_{k,1}  
\end{pmatrix}$ with real entries and 
$$ 0\le a_{ij}\le n+\alpha_{i+j-1}~\text{for all $1\le i\le k$ and $1\le j\le k-i+1$}, $$
  $$ \sum_{j=1}^m a_{ij}\le \sum_{j=1}^m a_{(i-1)j},~\text{for all $2\le i\le k$ and $1\le m\le k-i+1$}, $$
  $$ \sum_{i=1}^{m}a_{i,(m+1-i)}= n+\alpha_m ~\text{for all $1\le m\le k$}. $$
  or written as vector inequalities and equalities:
$$ \begin{pmatrix}
-1 & 1 \\
-1 & 1 & -1 & 1 \\
\vdots & & & \ddots & \ddots\\
-1 & 1 & \dots & \dots & -1 & 1 
\end{pmatrix}
\begin{pmatrix}
a_{(i-1),1}\\
a_{i,1}\\
a_{(i-1),2}\\
a_{i,2}\\
\vdots\\
a_{(i-1),(k-i+1)}\\
a_{i,(k-i+1)}
\end{pmatrix}\le
\begin{pmatrix}
0\\
\vdots\\
0
\end{pmatrix},~\text{for all $2\le i\le k$}, $$
$$ \begin{pmatrix}
1\\
 & 1 & 1 \\
 &  &  & 1 & 1 & 1 \\
 &  & & &  & & \ddots\\
 &  &  &  & & & 1 & 1 & \dots & 1
\end{pmatrix}
\begin{pmatrix}
a_{11}\\
a_{12}\\
a_{21}\\
a_{13}\\
a_{22}\\
a_{31}\\
\vdots\\
a_{1k}\\
\vdots\\
a_{(k-1),2}\\
a_{k,1}
\end{pmatrix}=
\begin{pmatrix}
n+\alpha_1\\
\vdots\\
n+\alpha_k
\end{pmatrix}. $$
$P_{k,n,\alpha}$ is a convex polytope in $\RR^{\binom{k+1}{2}}$. If $M$ is a set of partial matrices we write $M^{\ZZ}$ for its subset of partial matrices with only integer valued entries.
\begin{Proposition}\label{polytope}
Let $k,n\in\NN$ and $\alpha=(\alpha_1,...,\alpha_k)\in \NN_0^k$ with $\alpha_1\ge...\ge \alpha_k$. The polytope $P_{k,n,\alpha}$ is $\binom k2$-dimensional.
\end{Proposition}
\begin{proof}
$P_{k,n,\alpha}$ is contained in the affine subspace defined by 
$$ \begin{pmatrix}
1\\
 & 1 & 1 \\
 &  &  & 1 & 1 & 1 \\
 &  & & &  & & \ddots\\
 &  &  &  & & & 1 & 1 & \dots & 1
\end{pmatrix}
\begin{pmatrix}
a_{11}\\
a_{12}\\
a_{21}\\
a_{13}\\
a_{22}\\
a_{31}\\
\vdots\\
a_{1k}\\
\vdots\\
a_{(k-1),2}\\
a_{k,1}
\end{pmatrix}=
\begin{pmatrix}
n+\alpha_1\\
\vdots\\
n+\alpha_k
\end{pmatrix}$$
The above matrix has rank $k$. This implies that the affine subspace has dimension $\binom{k+1}{2}-k=\binom{k}{2}$.
Otherwise, $P_{k,n,\alpha}$ contains the following $\binom k2+1$ affine independent points. For every $i\in \{0,...,k-1\}$ we construct $\min\{1,i\}$ points starting with 
$$ \begin{pmatrix}
 n+\alpha_1 & 0 & 0 & \dots & 0 & n+\alpha_{k-i+1} & \dots & n+\alpha_k\\
 n+\alpha_2 &  0 & \dots & \dots & \dots & \dots & 0\\
 \vdots&\vdots \\
 n+\alpha_{k-i} & 0\\
 0&\vdots\\
 \vdots \\
 0  
\end{pmatrix}. $$
We get the next point by moving $n+\alpha_k$ diagonally left and down:
$$ \begin{pmatrix}
 n+\alpha_1 & 0 & 0 & \dots & 0 & n+\alpha_{k-i+1} & \dots & n+\alpha_{k-1} & 0\\
 n+\alpha_2 &  0 & \dots & 0 &0 &\dots &0 &  n+\alpha_k\\
 \vdots&\vdots\\
 n+\alpha_{k-i} & 0\\
 0&\vdots\\
 \vdots\\
 0  
\end{pmatrix}. $$
Now, we move $n+\alpha_{k-1}$ and $n+\alpha_{k}$ diagonally left and down:
$$ \begin{pmatrix}
 n+\alpha_1 & 0 & 0 & \dots & 0 & n+\alpha_{k-i+1} & \dots & n+\alpha_{k-2} & 0& 0\\
 n+\alpha_2 &  0 & \dots & 0 &0 & \dots &0 &  n+\alpha_{k-1} & 0\\
 n+\alpha_3 & 0 & \dots &0&0&\dots&0& n+\alpha_k\\
 \vdots&\vdots\\
 n+\alpha_{k-i} & 0\\
 0&\vdots\\
 \vdots\\
 0  
\end{pmatrix}. $$
We go on moving the rightmost nonzero values diagonally left and down until getting the point
$$ \begin{pmatrix}
 n+\alpha_1 & 0 & 0 & \dots & 0 & n+\alpha_{k-i+1} & \dots & 0 & 0& 0\\
 n+\alpha_2 &  0 & \dots & 0  & &n+\alpha_{k-i+2} &  0&\dots & 0\\
 n+\alpha_3 & 0 & \dots &0&0& n+\alpha_{k-i+3}&0&\dots\\
 \vdots & & & & &\vdots\\
 \vdots & 0 & 0&0&0& n+\alpha_k\\
 n+\alpha_{k-i} & 0 &\vdots\\
 0&\vdots\\
 \vdots\\
 0  
\end{pmatrix}. $$
We constructed $1+1+2+\cdots (k-1)=\binom k2 +1$ affine independent points lying in $P_{k,n,\alpha}$.
\end{proof}
\begin{Lemma}
Let $k,n\in\NN$ and $\alpha=(\alpha_1,...,\alpha_k)\in \NN_0^k$ with $\alpha_1\ge...\ge \alpha_k$. Then
$$ \prod_{i=1}^k s_{(n+\alpha_i)} = \sum_{A\in P_{k,n,\alpha}^{\ZZ}} s_{(|a_1|,...,|a_{k}|)}. $$
\end{Lemma}
\begin{proof}
The product $\prod_{i=1}^k s_{(n+\alpha_i)}$ is the homogeneous symmetric function $h_{(n+\alpha_1,...,n+\alpha_k)}$. It follows from the transition matrix between the basis of Schur functions and the basis of homogeneous symmetric functions that $h_{(n+\alpha_1,...,n+\alpha_k)}$ is the sum $\sum_{T} s_{\shape(T)}$ running over all semistandard Young tableaux $T$ of weight $(n+\alpha_1,...,n+\alpha_k)$. For every such tableau $T$ let $a_{i,j}(T)$ be the number of $(i+j-1)$'s in the $i$th row. Then the map given by
$$ T\mapsto \begin{pmatrix}
 a_{11}(T) & a_{12}(T) & a_{13}(T) & \dots & a_{1k}(T)\\
 a_{21}(T) &  a_{22}(T) & \dots & a_{2,(k-1)}(T)\\
 \vdots\\
 a_{(k-1),1}(T) & a_{(k-1),2}(T)\\
 a_{k,1}(T)  
\end{pmatrix} $$ 
is a bijection between the semistandard Young tableaux $T$ of weight $(n+\alpha_1,...,n+\alpha_k)$ and $P_{k,n,\alpha}^{\ZZ}$.
\end{proof}
 \begin{Lemma}\label{delta1hochk}
 Let $k\ge 1$ and $\alpha=(\alpha_1,...,\alpha_k)\in \NN_0^k$ with $\alpha_1\ge...\ge \alpha_k$. We consider the sequence $\{\prod_{i=1}^k s_{(n+\alpha_i)}\}_n\in\Lambda_{\NN}^{|\alpha|,k}$. For all $n\ge 1$, we have
  $$ \Delta^{(1^k)}\left(\prod_{i=1}^k s_{(n+\alpha_i)}\right) $$
  $$ =\sum_{A\in P_{k,n,\alpha}^{\ZZ}:~a_{k,1}=0} s_{(|a_1|,...,|a_{k-1}|)}. $$
 \end{Lemma}
 \begin{proof}
 It follows from the previous lemma that 
 $$ \Delta^{(1^k)}\left(\prod_{i=1}^k s_{(n+\alpha_i)}\right)= \sum_{A\in P_{k,n,\alpha}^{\ZZ}} s_{(|a_1|,...,|a_{k}|)}-\sum_{A\in P_{k,n-1,\alpha}^{\ZZ}} s_{(|a_1|+1,...,|a_{k}|+1)}.$$
 There is an injection $ P_{k,n-1,\alpha}^{\ZZ}\rightarrow  P_{k,n,\alpha}^{\ZZ}$ given by
 $$ \begin{pmatrix}
 a_{11} & a_{12} & a_{13} & \dots & a_{1k}\\
 a_{21} &  a_{22} & \dots & a_{2,(k-1)}\\
 \vdots\\
 a_{(k-1),1} & a_{(k-1),2}\\
 a_{k,1}  
\end{pmatrix}\mapsto
\begin{pmatrix}
 a_{11}+1 & a_{12} & a_{13} & \dots & a_{1k}\\
 a_{21}+1 &  a_{22} & \dots & a_{2,(k-1)}\\
 \vdots\\
 a_{(k-1),1}+1 & a_{(k-1),2}\\
 a_{k,1}+1  
\end{pmatrix}. $$
The matrices that are not hit by this map are those with a $0$ in the first column. This property is equivalent to $a_{k,1}=0$ because of $0\le a_{k,1}\le a_{k-1,1}\le...\le a_{11}$.
 \end{proof}
 We look at the polytope $\Delta^{(1^k)}P_{k,n,\alpha}:=\{A\in P_{k,n,\alpha}:~a_{k,1}=0\}$.
 \begin{Proposition}
  Let $k\ge 1$ and $\alpha=(\alpha_1,...,\alpha_k)\in \NN_0^k$ with $\alpha_1\ge...\ge \alpha_k$. $\Delta^{(1^k)}P_{k,n,\alpha}$ is a facet of $P_{k,n,\alpha}$.
 \end{Proposition}
 \begin{proof}
 $\Delta^{(1^k)}P_{k,n,\alpha}$ is a proper face of $P_{k,n,\alpha}$ because of $a_{k,1}= 0$ for all $A\in \Delta^{(1^k)}P_{k,n,\alpha}$ but not for every $A\in P_{k,n,\alpha}$ and $a_{k,1}\ge 0$ for all $A\in P_{k,n,\alpha}$. This face is $\left(\binom k2 -1\right)$-dimensional because the set of $\binom k2 +1$ many affine independent points of $P_{k,n,\alpha}$ given in the proof of \ref{polytope} contains exactly one point that does not lie in $\Delta^{(1^k)}P_{k,n,\alpha}$.
 \end{proof}
 \begin{Lemma}\label{delta1delta21}
 Let $k\ge 1$ and $\alpha=(\alpha_1,...,\alpha_k)\in \NN_0^k$ with $\alpha_1\ge...\ge \alpha_k$. We consider the sequence $\{\prod_{i=1}^k s_{(n+\alpha_i)}\}_n\in\Lambda_{\NN}^{|\alpha|,k}$. For all $n\ge 2$, we have 
  $$ \Delta^{(2,1^{k-2})}\Delta^{(1^k)}\left(\prod_{i=1}^k s_{(n+\alpha_i)}\right) $$
  $$ =\sum_{A\in P_{k,n,\alpha}^{\ZZ}:~a_{k,1}=0\wedge (a_{1k}=0\vee a_{k-1,1}=0)} s_{(|a_1|,...,|a_{k-1}|)}. $$
\end{Lemma}
 \begin{proof}
 \ref{delta1hochk} yields
 $$ \Delta^{(2,1^{k-2})}\Delta^{(1^k)}\prod_{i=1}^k s_{(n+\alpha_i)} $$
 $$ =\sum_{A\in P_{k,n,\alpha}^{\ZZ}:~a_{k,1}=0}s_{(|a_1|,...,|a_{k-1}|)} 
 -\sum_{A\in P_{k,n,\alpha}^{\ZZ}:~a_{k,1}=0}s_{(|a_1|+2,|a_2|+1,...,|a_{k-1}|+1)}.   $$
 There is an injective map $\{A\in P_{k,n-1,\alpha}^{\ZZ}:~a_{k,1}=0\}\mapsto \{A\in P_{k,n,\alpha}^{\ZZ}:~a_{k,1}=0\}$ given by
 $$
\begin{pmatrix}
 a_{11} & a_{12} & a_{13} & \dots & a_{1k}\\
 a_{21} & a_{22} & \dots & a_{2,(k-1)}\\
 \vdots\\
 a_{(k-1),1} & a_{(k-1),2}\\
 0  
\end{pmatrix}\mapsto
\begin{pmatrix}
 a_{11}+1 & a_{12} & a_{13} & \dots & a_{1k}+1\\
 a_{21}+1 &  a_{22} & \dots & a_{2,(k-1)}\\
 \vdots\\
 a_{(k-1),1}+1 & a_{(k-1),2}\\
 0  
\end{pmatrix}.
$$
 The matrices that are not hit are those with
$a_{i1}=0$ for a $1\le i\le k-1$ or $a_{1k}=0$. We can reduce the condition to $a_{1k}=0$ or $a_{k-1,1}=0$ because of $0\le a_{k-1,1}\le a_{k-2,1}\le...\le a_{11}$.
 \end{proof}
 Let $\Delta^{(2,1^{k-2})}\Delta^{(1^k)}P_{k,n,\alpha}=\{A\in \Delta^{(1^k)}P_{k,n,\alpha}:~a_{1k}=0\vee a_{k-1,1}=0\}$.
 \begin{Proposition}
 Let $k\ge 1$ and $\alpha=(\alpha_1,...,\alpha_k)\in \NN_0^k$ with $\alpha_1\ge...\ge \alpha_k$. $\Delta^{(2,1^{k-2})}\Delta^{(1^k)}P_{k,n,\alpha}$ is the union of the two facets $\{A\in \Delta^{(1^k)}P_{k,n,\alpha}:~ a_{1,k}=0\}$ and $\{A\in \Delta^{(1^k)}P_{k,n,\alpha}:~a_{k-1,1}=0\}$ of $\Delta^{(1^k)}P_{k,n,\alpha}$.
 \end{Proposition}
 \begin{proof}
 The two sets are proper subsets of $\Delta^{(1^k)}P_{k,n,\alpha}$ and faces because of $a_{1,k}\ge 0$ and $a_{k-1,1}\ge 0$ for all $A\in \Delta^{(1^k)}P_{k,n,\alpha}$. $\{A\in \Delta^{(1^k)}P_{k,n,\alpha}:~ a_{k-1,1}=0\}$ is a facet of $\Delta^{(1^k)}P_{k,n,\alpha}$ because the set of $\binom k2 +1$ many affine independent points of $P_{k,n,\alpha}$ given in the proof of \ref{polytope} contains exactly two points with $a_{k-1,1}>0$. For $\{A\in \Delta^{(1^k)}P_{k,n,\alpha}:~a_{1,k}=0\}$ we slightly modify the list of affine independent points given in the proof of \ref{polytope}. For every $i\in \{2,...,k-2\}$ we take the $i$ points 
 $$ \begin{pmatrix}
 n+\alpha_1 & 0 & 0 & \dots & 0 & n+\alpha_{k-i+1} & \dots & n+\alpha_{k-1} & 0\\
 n+\alpha_2 &  0 & \dots & 0 &0 &\dots &0 &  n+\alpha_k\\
 \vdots&\vdots\\
 n+\alpha_{k-i} & 0\\
 0&\vdots\\
 \vdots\\
 0  
\end{pmatrix}, $$
$$ \begin{pmatrix}
 n+\alpha_1 & 0 & 0 & \dots & 0 & n+\alpha_{k-i+1} & \dots & n+\alpha_{k-2} & 0& 0\\
 n+\alpha_2 &  0 & \dots & 0 &0 & \dots &0 &  n+\alpha_{k-1} & 0\\
 n+\alpha_3 & 0 & \dots &0&0&\dots&0& n+\alpha_k\\
 \vdots&\vdots\\
 n+\alpha_{k-i} & 0\\
 0&\vdots\\
 \vdots\\
 0  
\end{pmatrix},..., $$
$$ \begin{pmatrix}
 n+\alpha_1 & 0 & 0 & \dots & 0 & n+\alpha_{k-i+1} & \dots & 0 & 0& 0\\
 n+\alpha_2 &  0 & \dots & 0  & &n+\alpha_{k-i+2} &  0&\dots & 0\\
 n+\alpha_3 & 0 & \dots &0&0& n+\alpha_{k-i+3}&0&\dots\\
 \vdots & & & & &\vdots\\
 \vdots & 0 & 0&0&0& n+\alpha_k\\
 n+\alpha_{k-i} & 0 &\vdots\\
 0&\vdots\\
 \vdots\\
 0  
\end{pmatrix} $$
and  
$$ \begin{pmatrix}
 n+\alpha_1 & 0  & \dots &0 & 1 & n+\alpha_{k-i+1} & \dots & 0 & 0& 0\\
 n+\alpha_2 &  0 & \dots & 0  & &n+\alpha_{k-i+2} &  0&\dots & 0\\
 n+\alpha_3 & 0 & \dots &0&0& n+\alpha_{k-i+3}&0&\dots\\
 \vdots & & & & &\vdots\\
 \vdots & 0 & 0&0&0& n+\alpha_k\\
 n+\alpha_{k-i}-1 & 0 &\vdots\\
 0&\vdots\\
 \vdots\\
 0  
\end{pmatrix}. $$
We additionally take the $k-1$ points
$$ \begin{pmatrix} 
   n+\alpha_1 & n+\alpha_2 & \dots &\dots & \dots & n+\alpha_{k-1} &0\\
   0 &0&\dots & 0 & 0 &  n+\alpha_{k}\\
   0&\dots & \dots &0&  0\\
   \vdots\\
   0&0&0&0\\
   0&0&0\\
   0&0\\
   0 
  \end{pmatrix},\dots, 
  \begin{pmatrix} 
   n+\alpha_1 & n+\alpha_2 & \dots &0& 0 & 0 &0\\
   0 &n+\alpha_3&\dots & 0 & 0 &  0\\
   0&\vdots & \dots &0&  0\\
   \vdots\\
   0&n+\alpha_{k-2}&0&0\\
   0&n+\alpha_{k-1}&0\\
   0&n+\alpha_k\\
   0 
  \end{pmatrix}$$
  and
  $$
  \begin{pmatrix} 
   n+\alpha_1 & n+\alpha_2-1 & \dots &0& 0 & 0 &0\\
   1 &n+\alpha_3&\dots & 0 & 0 &  0\\
   0&\vdots & \dots &0&  0\\
   \vdots\\
   0&n+\alpha_{k-2}&0&0\\
   0&n+\alpha_{k-1}&0\\
   0&n+\alpha_k\\
   0 
  \end{pmatrix}.$$
  These are $2+3+...+(k-2)+(k-1)=\binom{k}{2}-1$ many affine independent points.
 \end{proof}
 Now we can prove \ref{main4}.
 \begin{proof}[Proof of \ref{main4}]
 $(a)$
 It follows from \ref{delta1delta21} that for all $n\ge 2$:
  $$ \Delta^{(2)}\Delta^{(1,1)}s_{(n+\alpha_1)}s_{(n)} $$
  $$ =\sum_{A} s_{(|a_1|,|a_2|)} $$
  running over all
    $$ A=\begin{pmatrix}
 a_{11} & a_{12}\\
 a_{21}
\end{pmatrix}\in P_{2,n,\alpha}^{\ZZ}   $$
with
 $$ a_{21}=0,~a_{11}= n+\alpha_{1},~ a_{12}= n, $$
$$
a_{11}=0~\text{or}~a_{12}=0.$$
But this cannot be for $n>0$. Therefore the sum is zero.\\
We show the minimality of the set $\{\Delta^{(2)},\Delta^{(1,1)}\}$ next.
By \ref{delta1hochk}, we have 
$$ \Delta^{(1,1)}s_{(n+\alpha_1)}s_{(n)} $$
  $$ =\sum_{A} s_{(|a_1|,|a_2|)} $$
  running over all
    $$ A=\begin{pmatrix}
 a_{11} & a_{12}\\
 a_{21}
\end{pmatrix}\in P_{2,n,\alpha}^{\ZZ}   $$
with
 $$ a_{21}=0,~a_{11}= n+\alpha_{1},~ a_{12}= n. $$
 This is the matrix 
 $$ A= \begin{pmatrix}
 n+\alpha_1 & n\\
 0
\end{pmatrix}$$
and the above sum is not zero. The term $\Delta^{(2)}s_{(n+\alpha_1)}s_{(n)}$ does also not equal zero because the sum
$$s_{(n+\alpha_1)}s_{(n)}=\sum_{A\in P_{2,n,\alpha}^{\ZZ}} s_{(|a_1|,|a_2|)} $$
has the summand $s_{(n+\alpha_1,n)}$ that is not cancelled by applying $\Delta^{(2)}$ because $|a_1|\ge n-1+\alpha_1$ for all $A\in P_{2,n-1,\alpha}^{\ZZ}$. 
\\
$(b)$
 It follows from \ref{delta1delta21} that for all $n\ge 2$:
 $$ \Delta^{(2,1)}\Delta^{(1,1,1)}s_{(n+\alpha_1)}s_{(n+\alpha_2)}s_{(n)} $$
  $$ =\sum_{A} s_{(|a_1|,|a_2|)} $$
  running over all
    $$ A=\begin{pmatrix}
 a_{11} & a_{12} & a_{13}\\
 a_{21} &  a_{22}\\
 a_{31}
\end{pmatrix}\in P_{3,n,\alpha}^{\ZZ}   $$
with
$$ a_{31}=0, $$
$$
a_{13}=0~\text{or}~
 a_{21}=0.  
$$
We want to apply the map $\Delta^{(3)}$ next. We treat the two sets $Q_{n,21}=\{A\in P_{3,n-1,\alpha}^{\ZZ}~|~a_{31}=0\wedge a_{21}=0\}$ and $Q_{n,13}=\{A\in P_{3,n-1,\alpha}^{\ZZ}~|~a_{31}=0\wedge a_{13}=0\wedge a_{21}>0\}$ separately.
There is an injection $Q_{n-1,21}\rightarrow Q_{n,21}$ given by
$$ \begin{pmatrix}
 a_{11} & a_{12} & a_{13}\\
 0 &  a_{22}\\
 0
\end{pmatrix}\mapsto
\begin{pmatrix}
 a_{11}+1 & a_{12}+1 & a_{13}+1\\
 0&  a_{22}\\
 0
\end{pmatrix}. $$
The only matrix that is not hit by this map is 
$\begin{pmatrix}
 n+\alpha_1 & n+\alpha_2 & 0\\
 0&  n\\
 0
\end{pmatrix}$. It follows
$$ \Delta^{(3)}\Delta^{(2,1)}\Delta^{(1,1,1)}s_{(n+\alpha_1)}s_{(n+\alpha_2)}s_{(n)}$$ $$=
s_{(2n+|\alpha|,n)}+\sum_{A\in Q_{n,13}}s_{(|a_1|,|a_2|)}-\sum_{B\in Q_{n-1,13}} s_{(|b_1|,|b_2|)} $$
 We have
$$ Q_{n,13}=\left\{
\begin{pmatrix}
 n+\alpha_1 & n+\alpha_2-a_{21} & 0\\
 a_{21} &  n\\
 0
\end{pmatrix}~|~a_{21}\in \left\{1,...,\left\lfloor \frac{n+|\alpha|}{2} \right\rfloor\right\}\right\}. $$
It follows that for every $B\in Q_{n-1,13}$ there is exactly one $A\in Q_{n,13}$ with $(|a_1|,|a_2|)=(|b_1|+3,|b_2|)$. It is the matrix $A$ with $a_{21}=b_{21}-1$. There is one matrix in $Q_{n-1,13}$ and one or two matrices in $Q_{n,13}$ not involved in this correspondence depending on the parity of $n+|\alpha|$. These matrices are 
$$ \begin{pmatrix}
 n-1+\alpha_1 & n-2+\alpha_2 & 0\\
 1 &  n-1\\
 0
\end{pmatrix}\in Q_{n-1,13}, $$
$$ \begin{pmatrix}
 n+\alpha_1 & (n+\alpha_2-\alpha_1)/2+1 & 0\\
 (n+|\alpha|)/2-1 &  n\\
 0
\end{pmatrix},
\begin{pmatrix}
 n+\alpha_1 & (n+\alpha_2-\alpha_1)/2 & 0\\
 (n+|\alpha|)/2 &  n\\
 0
\end{pmatrix} \in Q_{n,13} $$ 
$$\text{if $n+|\alpha|$ is even}, $$
$$ \begin{pmatrix}
 n+\alpha_1 & (n+\alpha_2-\alpha_1+1)/2 & 0\\
 (n+|\alpha|-1)/2 &  n\\
 0
\end{pmatrix}\in Q_{n,13}~\text{if $n+|\alpha|$ is odd}. $$
 Now in 
$\Delta^{(3)}\Delta^{(2,1)}\Delta^{(1,1,1)}(s_{(n+\alpha_1)}s_{(n+\alpha_2)}s_{(n)})$, the Schur function $s_{(2n+|\alpha|,n)}$ from before is subtracted and what is left is
$$ s_{(3n+|\alpha|)/2+1,(3n+|\alpha|)/2-1)}
+s_{((3n+|\alpha|)/2,(3n+|\alpha|)/2)},~\text{if $n+|\alpha|$ is even} $$
$$s_{((3n+|\alpha|+1)/2,(3n+|\alpha|-1)/2)},~\text{if $n+|\alpha|$ is odd}. $$
Applying $\Delta_2^{(3,3)}$ to this yields $0$.\\
We show the minimality of the set $\{\Delta^{(3,3)}_2,\Delta^{(3)},\Delta^{(2,1)},\Delta^{(1^3)}\}$ next. We see above that $\Delta^{(3)}\Delta^{(2,1)}\Delta^{(1,1,1)}(s_{(n+\alpha_1)}s_{(n+\alpha_2)}s_{(n)})$ is not zero. The term $\Delta^{(3,3)}_2\Delta^{(3)}\Delta^{(2,1)}(s_{(n+\alpha_1)}s_{(n+\alpha_2)}s_{(n)})$ is not zero because the summand $s_{(n+\alpha_1,n+\alpha_2,n)}$ in the Schur function expansion of $s_{(n+\alpha_1)}s_{(n+\alpha_2)}s_{(n)}$ is not cancelled by $\Delta^{(3)}$ or $\Delta^{(2,1)}$ because $|a_1|\ge n-1+\alpha_2$ for all $A\in P_{3,n-1,\alpha}^{\ZZ}$ and it is not cancelled by $\Delta^{(3,3)}_2$ because $|a_1|\ge n-2+\alpha_2$ for all $A\in P_{3,n-2,\alpha}^{\ZZ}$. The term $\Delta^{(3,3)}_2\Delta^{(2,1)}\Delta^{(1^3)}(s_{(n+\alpha_1)}s_{(n+\alpha_2)}s_{(n)})$ is not zero because the summand $s_{(3n+|\alpha|)}$ in the Schur function expansion of $s_{(n+\alpha_1)}s_{(n+\alpha_2)}s_{(n)}$ is not cancelled. In order to show that the term $\Delta^{(3,3)}_2\Delta^{(3)}\Delta^{(1^3)}(s_{(n+\alpha_1)}s_{(n+\alpha_2)}s_{(n)})$ is not zero we show that the multiplicity of $s_{(2n+|\alpha|,n)}$ in its Schur function decomposition is not zero. We denote the multiplicity of a Schur function $s_{\lambda}$ in a symmetric function $f$ by $\mult(\lambda,f)$. Now, we have
$$ \mult((2n+|\alpha|,n),\Delta^{(3,3)}_2\Delta^{(3)}\Delta^{(1^3)}(s_{(n+\alpha_1)}s_{(n+\alpha_2)}s_{(n)})) $$
$$ =\mult((2n+|\alpha|,n),\Delta^{(1^3)}(s_{(n+\alpha_1)}s_{(n+\alpha_2)}s_{(n)}))$$ $$-\mult((2n-3+|\alpha|,n),\Delta^{(1^3)}(s_{(n-1+\alpha_1)}s_{(n-1+\alpha_2)}s_{(n-1)})) $$
$$ -\mult((2n-3+|\alpha|,n-3),\Delta^{(1^3)}(s_{(n-2+\alpha_1)}s_{(n-2+\alpha_2)}s_{(n-2)}))$$ $$+\mult((2n-6+|\alpha|,n-3),\Delta^{(1^3)}(s_{(n-3+\alpha_1)}s_{(n-3+\alpha_2)}s_{(n-3)})). $$
The four involved numbers count in this order the number of matrices in $\Delta^{(1^4)}P_{4,n,\alpha}^{\ZZ}$ of the form
$$ \begin{pmatrix}
 n+\alpha_1 & n+\alpha_2-a_{21} & a_{21}\\
 a_{21} &  n-a_{21}\\
 0
\end{pmatrix}~\text{for all $a_{21}\in \{0,...,n\}$}, $$
$$ \begin{pmatrix}
 n-1+\alpha_1 & n-1+\alpha_2-a_{21} & a_{21}-1\\
 a_{21} &  n-a_{21}\\
 0
\end{pmatrix}~\text{for all $a_{21}\in \{1,...,n-1+\min\{1,\alpha_2\}\}$}, $$
$$ \begin{pmatrix}
 n-2+\alpha_1 & n-2+\alpha_2-a_{21} & a_{21}+1\\
 a_{21} &  n-3-a_{21}\\
 0
\end{pmatrix}~\text{for all $a_{21}\in \{0,...,n-3\}$}, $$
$$ \begin{pmatrix}
 n-3+\alpha_1 & n-3+\alpha_2-a_{21} & a_{21}\\
 a_{21} &  n-3-a_{21}\\
 0
\end{pmatrix}~\text{for all $a_{21}\in \{0,...,n-3\}$}. $$
It follows
$$ \mult((2n+|\alpha|,n),\Delta^{(3,3)}_2\Delta^{(3)}\Delta^{(1^3)}(s_{(n+\alpha_1)}s_{(n+\alpha_2)}s_{(n)}))$$ $$=n+1-(n-1+\min\{1,\alpha_2\})-(n-2)+(n-2)\ge 1. $$
  \end{proof}
  The next lemmas are dedicated to the statement \ref{vierfachhom} about fourfold products.
\begin{Lemma}
We have 
$$ (\Delta^{(2,1,1)})^2\Delta^{(1^4)} (s_{(n+\alpha_1)}s_{(n+\alpha_2)}s_{(n+\alpha_3)}s_{(n)}) $$
$$ = \sum_{A\in \Delta^{(2,1,1)}\Delta^{(1^4)}P_{4,n,\alpha}^{\ZZ}:~a_{13}=0\lor a_{32}=0\lor a_{31}=a_{21}\lor a_{21}+a_{22}+a_{23}=a_{11}+a_{12}+a_{13}}s_{(|a_1|,|a_2|,|a_3|)}. $$
 \end{Lemma}
 \begin{proof}
  \ref{delta1delta21} yields
$$\Delta^{(2,1,1)}\Delta^{(1^4)}(s_{(n+\alpha_1)}s_{(n+\alpha_2)}s_{(n+\alpha_3)}s_{(n)})=\sum_{A\in \Delta^{(2,1,1)}\Delta^{(1^4)}P_{4,n,\alpha}^{\ZZ}}s_{(|a_1|,|a_2|,|a_3|)}.$$
 There is an injective map $\Delta^{(2,1,1)}\Delta^{(1^4)}P_{4,n-1,\alpha}^{\ZZ}\rightarrow \Delta^{(2,1,1)}\Delta^{(1^4)}P_{4,n,\alpha}^{\ZZ}$ given by
 $$  \begin{pmatrix}
 a_{11} & a_{12} & a_{13} & a_{14}\\
 a_{21} & a_{22} & a_{23}\\
 a_{31} & a_{32}\\
 0  
\end{pmatrix}\mapsto
\begin{pmatrix}
 a_{11}+1 & a_{12} & a_{13}+1 & a_{14}\\
 a_{21}+1 & a_{22} & a_{23}\\
 a_{31} & a_{32}+1\\
 0   
\end{pmatrix}. $$
The matrices that are not hit are those with $a_{13}=0\lor a_{21}=0\lor a_{32}=0\lor a_{31}=a_{21}\lor a_{21}+a_{22}+a_{23}=a_{11}+a_{12}+a_{13}$. $a_{21}=0$ implies $a_{31}=a_{21}$ because of $0\le a_{31}\le a_{21}$.
\end{proof}
Let $(\Delta^{(2,1,1)})^2\Delta^{(1^4)}P_{4,n,\alpha}:=\{A\in \Delta^{(2,1,1)}\Delta^{(1^4)}P_{4,n,\alpha}:~a_{13}=0\lor a_{32}=0\lor a_{31}=a_{21}\lor a_{21}+a_{22}+a_{23}=a_{11}+a_{12}+a_{13}\}$. We know so far that   
$$P_{4,n,\alpha}\supseteq \Delta^{(1^4)}P_{4,n,\alpha}\supseteq \Delta^{(2,1,1)}\Delta^{(1^4)}P_{4,n,\alpha}$$ 
is a sequence of unions of faces of $P_{4,n,\alpha}$ with corresponding dimensions 
$$ 6>5>4 $$
where we say that the dimension of a union of polytopes is the maximum of the dimensions of the polytopes.
Here, we have again that the subset $(\Delta^{(2,1,1)})^2\Delta^{(1^4)}P_{4,n,\alpha}\subseteq \Delta^{(2,1,1)}\Delta^{(1^4)}P_{4,n,\alpha}$ is a union of faces. Its dimension is $3$ because the face 
$\{A\in P_{4,n,\alpha}:~a_{41}=a_{14}=a_{13}=0\}\subseteq (\Delta^{(2,1,1)})^2\Delta^{(1^4)}P_{4,n,\alpha}$ contains the $4$ affine independent points
$$  \begin{pmatrix}
 n+\alpha_1 & n+\alpha_2 & 0 & 0\\
 0 & n+\alpha_3 & n\\
 0 & 0\\
 0   
\end{pmatrix},
\begin{pmatrix}
 n+\alpha_1 & n+\alpha_2 & 0 & 0\\
 0 & n+\alpha_3 & 0\\
 0 & n\\
 0   
\end{pmatrix}, $$  $$\begin{pmatrix}
 n+\alpha_1 & n+\alpha_2-1 & 0 & 0\\
 1 & n+\alpha_3 & 0\\
 0 & n\\
 0   
\end{pmatrix},
\begin{pmatrix}
 n+\alpha_1 & n+\alpha_2-1 & 0 & 0\\
 1 & n+\alpha_3-1 & n\\
 1 & 0\\
 0   
\end{pmatrix}. $$
\begin{Lemma}
We have 
$$ \Delta^{(2,2)}(\Delta^{(2,1,1)})^2\Delta^{(1^4)}(s_{(n+\alpha_1)}s_{(n+\alpha_2)}s_{(n+\alpha_3)}s_{(n)}) $$
$$ =\sum_{A\in (\Delta^{(2,1,1)})^2\Delta^{(1^4)}P_{4,n,\alpha}^{\ZZ}:~a_{12}=0\lor a_{22}=0\lor a_{23}=0\lor a_{21}+a_{22}=a_{11}+a_{12}\lor a_{31}+a_{32}=a_{21}+a_{22}} s_{(|a_1|,|a_2|,|a_3|)}. $$
\end{Lemma}
\begin{proof}
There is an injection $(\Delta^{(2,1,1)})^2\Delta^{(1^4)}P_{4,n-1,\alpha}^{\ZZ}\rightarrow 
(\Delta^{(2,1,1)})^2\Delta^{(1^4)}P_{4,n,\alpha}^{\ZZ}$ given by 
 $$\begin{pmatrix}
 a_{11} & a_{12} & a_{13} & a_{14}\\
 a_{21} & a_{22} & a_{23}\\
 a_{31} & a_{32}\\
 a_{41}  
\end{pmatrix}\mapsto
\begin{pmatrix}
 a_{11}+1 & a_{12}+1 & a_{13} & a_{14}\\
 a_{21} & a_{22}+1 & a_{23}+1\\
 a_{31} & a_{32}\\
 a_{41}   
\end{pmatrix}.$$
The matrices that are not hit are those with $a_{12}=0$ or $a_{22}=0$ or $a_{23}=0$ or $a_{21}+a_{22}=a_{11}+a_{12}$ or $a_{31}+a_{32}=a_{21}+a_{22}$.
\end{proof}
Let $\Delta^{(2,2)}(\Delta^{(2,1,1)})^2\Delta^{(1^4)}P_{4,n,\alpha}:=\{A\in (\Delta^{(2,1,1)})^2\Delta^{(1^4)}P_{4,n,\alpha}:~a_{12}=0\lor a_{22}=0\lor a_{23}=0\lor a_{21}+a_{22}=a_{11}+a_{12}\lor a_{31}+a_{32}=a_{21}+a_{22}\}$. The face $\{A\in P_{4,n,\alpha}:~a_{41}=a_{14}=a_{13}=a_{23}=0\}\subseteq \Delta^{(2,2)}(\Delta^{(2,1,1)})^2\Delta^{(1^4)}P_{4,n,\alpha}$ has dimension $2$ because it contains the $3$ affine independent points
$$ 
 \begin{pmatrix}
 n+\alpha_1 & n+\alpha_2 & 0 & 0\\
 0 & n+\alpha_3 & 0\\
 0 & n\\
 0   
\end{pmatrix},
 \begin{pmatrix}
 n+\alpha_1 & n+\alpha_2-1 & 0 & 0\\
 1 & n+\alpha_3 & 0\\
 0 & n\\
 0   
\end{pmatrix},$$  $$
 \begin{pmatrix}
 n+\alpha_1 & n+\alpha_2-1 & 0 & 0\\
 1 & n+\alpha_3-1 & 0\\
 1 & n\\
 0   
\end{pmatrix}. $$
We summarize that 
$$P_{4,n,\alpha}\supseteq \Delta^{(1^4)}P_{4,n,\alpha}\supseteq \Delta^{(2,1,1)}\Delta^{(1^4)}P_{4,n,\alpha}\supseteq (\Delta^{(2,1,1)})^2\Delta^{(1^4)}P_{4,n,\alpha}\supseteq \Delta^{(2,2)}(\Delta^{(2,1,1)})^2\Delta^{(1^4)}P_{4,n,\alpha}$$ 
is a sequence of unions of faces of $P_{4,n,\alpha}$ with corresponding dimensions 
$$ 6>5>4>3>2. $$


\begin{thebibliography}{99}
 
 \bibitem{CF}
T. Church, B. Farb, Representation theory and homological stability, {\it Advances in Mathematics} {\bf 245} (2013) 250--314
 
 \bibitem{M}
I.G. Macdonald, Symmetric Functions and Hall Polynomials, 2. ed.,
 Clarendon Press, Oxford, 1995
 
 \end{thebibliography}
\end{document}